\documentclass[11pt]{article}

\usepackage[letterpaper,margin=1.05in]{geometry}
\usepackage{amsmath,amssymb,amsthm,mathtools}
\usepackage{microtype}
\usepackage{enumitem}
\usepackage[colorlinks=true,linkcolor=blue,citecolor=blue,urlcolor=blue]{hyperref}

\newtheorem{theorem}{Theorem}[section]
\newtheorem{lemma}[theorem]{Lemma}
\newtheorem{proposition}[theorem]{Proposition}
\newtheorem{corollary}[theorem]{Corollary}
\theoremstyle{definition}
\newtheorem{definition}[theorem]{Definition}
\theoremstyle{remark}

\newcommand{\N}{\mathbb{N}}

\newcommand{\R}{\mathbb{R}}
\newcommand{\C}{\mathbb{C}}
\newcommand{\Log}{\operatorname{Log}}
\newcommand{\Arg}{\operatorname{Arg}}
\newcommand{\len}{\operatorname{len}}
\newcommand{\E}{\mathcal{E}}

\title{Tag-system computation using only powers and principal logarithms}
\author{Trevor Cappallo\thanks{\texttt{tc@by.tc}. An earlier version of the
construction was posted on Mathematics Stack Exchange in December 2023:
\url{https://math.stackexchange.com/questions/4827573/}.}}
\date{September 2026}

\begin{document}
\maketitle

\begin{abstract}
I show that one exact step of any tag system can be written as a single fixed
expression using only powers and logarithms.  A tag state is represented as an
integer; evaluating the expression yields the integer encoding of the next
configuration, and feeding the answer back into the same expression follows
the computation step by step.  The
discrete part of the update comes from leveraging the wraparound in the principal complex
logarithm, which effectively serves the same purpose as floor, mod, or sin in comparable approaches.
The iterating expressions use
only $x^y$ and $\log_b a$: it turns out that these two binary operations are sufficient, with no need for constants.
It follows that reachability
below a fixed threshold is undecidable for one fixed expression.  As a
separate application, the same mechanism represents addition, $2^x$, and
remainder with the convention $x\bmod0=x$, and hence every Kalm\'ar elementary
function.
\end{abstract}

\section{Introduction}

A tag system is a particularly simple kind of program.  It reads the first
symbol of a word, deletes a fixed number of symbols, and appends one of a
finite list of productions.  We store that leading symbol as the least
significant digit of an integer, which can then be manipulated with basic arithmetic:
shift off some digits, inspect the first digit, and attach the
chosen production at the other end.

The novelty here is that the same update can be written using only the operations
$x^y$ and $\log_ba$.  There is no floor, remainder, or conditional in the
final expression.  Instead, I use the fact that the principal complex
logarithm discards whole turns around the origin.  The number of discarded
turns is an integer---the unwinding number of Corless and Jeffrey
\cite{corless}---and for our purposes, it can be made to act as the floor operation.  Once floor
is available, the rest is bookkeeping.  The fully expanded expression is
more or less inscrutable, so I first write the update with ordinary arithmetic and then give
a mechanical translation back to powers and logarithms.

 The primary claim is about
\emph{iteration}: one evaluation performs one machine step, so repeated
evaluation follows the machine for as long as it runs.  In this sense the construction
is computationally universal, able to perform any computation a Turing machine could,
provided unbounded iteration and storage of the intermediate states.

Every binary tag system has an
exact step formula on every nonhalting binary word. The same construction
works for a finite alphabet; again, all displayed arithmetic and constants can be
eliminated, and reachability below a fixed threshold is undecidable for one
fixed term.  Additionally, Prunescu, Sauras-Altuzarra, and Shunia prove that the
Kalm\'ar elementary functions are generated by
$\{x+y,x\bmod y,2^x\}$, with $x\bmod0=x$ \cite{psa}.  Representing those
three operations in the form described in this paper gives the composition corollary they
suggested in correspondence.

The corollary in Appendix~\ref{sec:kalmar} belongs to the arithmetic-term and substitution-basis
program of Mazzanti and Marchenkov \cite{mazzanti,marchenkov}, most recently
sharpened by Prunescu, Sauras-Altuzarra, and Shunia.  Campagnolo, Moore, and
Costa give a complementary analog characterization of the Kalm\'ar elementary
functions: their restricted GPAC senses inequalities differentiably, whereas
the branch-cut mechanism here does so discontinuously \cite{campagnolo}.
Odrzywo{\l}ek has also
shown that the single binary operator
$\operatorname{eml}(x,y)=e^x-\ln y$, together with the constant $1$, generates
the standard repertoire of elementary functions \cite{odrzywolek}.  That
primitive has subtraction built in and concerns finite function generation;
the present language has no additive primitive, and the main claim concerns
iteration.

Analytic one-dimensional maps are already known to be universal under
iteration, and robust analytic examples are known as well \cite{koiran,graca}.
I am claiming a narrower point: namely, that one can execute an explicit tag update in the small syntax $\{x^y,\log_ba\}$,
with the branch convention, endpoints, and domains made explicit.  The proof
follows the informal derivation: first the integer-selection trick, then the
tag update, then iteration, and finally removal of the arithmetic shorthand.
The Kalm\'ar composition result is logically separate and is given in
Appendix~\ref{sec:kalmar}.  An alternative cyclic-tag formulation appears in
Appendix~\ref{sec:cyclic}.  Appendix~\ref{sec:collatz} gives a direct,
constant-free expression for a Collatz step.

\section{Principal-value log--exp terms}

The construction depends on one convention that cannot be left implicit.
Whenever a logarithm or power passes through a complex value, I use its
standard principal value.  In particular, the negative real axis is included
in the domain of $\Log$, even though $\Log$ jumps there.

Throughout, $\N=\{0,1,2,\ldots\}$.  For $z\ne0$ set
\[
    \Log z=\log|z|+i\Arg z, \qquad \Arg z\in(-\pi,\pi].
\]
Thus $\Log(-1)=i\pi$.  This is the principal value on $\C^\times$,
including the value on the negative real axis, where it is discontinuous.
For $x\ne0$ and for $a,b\ne0$ with $b\ne1$, define
\[
    x^y:=e^{y\Log x},
    \qquad
    \log_b a:=\frac{\Log a}{\Log b}.
\]

A \emph{log--exp term} is a \textbf{finite} term built from the two binary operations $x^y$ and
$\log_b a$, along with one or more variables and/or constants.
Such a term denotes a partial function: an input is outside its
domain if any subterm asks for one of the excluded logarithms or bases.

For readability, the next few sections display ordinary field arithmetic
and fixed integer constants.  Lemma~\ref{lem:compile} is a structural
translation of every such display into the two-operation language.

\section{The branch-cut trick}\label{sec:gadget}

The key observation is easiest to picture on the unit circle.  The quantity
$e^{2\pi i u}$ remembers only the fractional part of $u$.  Taking the
principal logarithm chooses one representative angle, and recovering the
discarded whole turns gives an integer.  The shift and orientation below
place the jump where it needs to be, including the value at the jump itself.

For $u\in\R$ define
\begin{equation}\label{eq:J}
    J(u):=u-\frac12+
    \frac{\Log\!\left(e^{2\pi i(1/2-u)}\right)}{2\pi i}.
\end{equation}

\begin{lemma}[Exact floor]\label{lem:floor}
For every real $u$, $J(u)=\lfloor u\rfloor$.
\end{lemma}

\begin{proof}
Write $u=n+f$, with $n=\lfloor u\rfloor$ and $f\in[0,1)$.  Since $n$ is an
integer,
\[
 e^{2\pi i(1/2-u)}=e^{2\pi i(1/2-f)}.
\]
The angle $2\pi(1/2-f)$ already lies in $(-\pi,\pi]$, including $\pi$ when
$f=0$.  Thus
\[
 \Log\!\left(e^{2\pi i(1/2-u)}\right)=2\pi i(1/2-f).
\]
Substitution into \eqref{eq:J} gives
$J(u)=n+f-1/2+(1/2-f)=n$.
\end{proof}

For example, $J(3.2)=3$, $J(3)=3$, and $J(-3.2)=-4$.  At an integer input,
the inner exponential is $-1$, whose principal logarithm is $i\pi$; this is
what gives the right value at the endpoint.

All identities in this paper use exact principal-value semantics.  No claim
of numerical stability is intended: at integral $u$, $J(u)$ deliberately
evaluates on the branch cut, and floating-point rounding near that endpoint
can change the result.

As for being constant-free, note that the constants $e,\pi,i$ in \eqref{eq:J} are only expository.
With the same principal-value convention,
\begin{equation}\label{eq:J-pure}
  J(u)=u-\frac12+\frac12
  \log_{-1}\!\left((-1)^{1-2u}\right).
\end{equation}
Indeed, $(-1)^{1-2u}=e^{2\pi i(1/2-u)}$ and $\Log(-1)=i\pi$, so \eqref{eq:J-pure} is exactly \eqref{eq:J}.

\section{Exact tag-system steps}\label{sec:tag}

A tag system $T=(v,\Sigma,P)$ consists of an integer deletion number
$v\ge2$, a finite alphabet $\Sigma$, and a production map
$P:\Sigma\to\Sigma^*$.  From a word $w=w_1\cdots w_n$ with $n\ge v$, one
step deletes $w_1\cdots w_v$ and appends $P(w_1)$.  Under the most commonly used convention, the system halts when
$n<v$.

It has long been established that tag systems with deletion number $2$ can be universal
\cite{cocke}. While the halting and reachability problems for binary $2$-tag
systems---TS(2, 2) in the literature---are, in fact, decidable \cite{demol}, binary tag systems with
larger deletion numbers can be universal \cite{neary}. 

Choose an alphabet of size $m\ge2$ with
$\Sigma=\{0,1,\ldots,m-1\}$.  The front of a word will occupy the least
significant digit.

\begin{definition}[Positional encoding]\label{def:encoding}
For $w=w_1\cdots w_n\in\Sigma^*$, define
\begin{equation}\label{eq:encoding}
  c_m(w):=m^n+\sum_{r=1}^n w_r m^{r-1}.
\end{equation}
The leading base-$m$ digit $1$ is a marker; in particular,
$c_m(\varepsilon)=1$.
\end{definition}

The ordering is backwards relative to the way the word is typically written,
but this makes deletion
an integer right shift.  For example, $00101$ is stored with data bits $10100$
and a marker $1$, giving $110100_2=52$.  The marker is essential: it preserves
zeroes at the far end, so $000$ does not collapse to the empty word.

The inequalities
\begin{equation}\label{eq:length-bounds}
   m^n\le c_m(w)<2m^n\le m^{n+1}
\end{equation}
show that $\lfloor\log_m c_m(w)\rfloor=n$.  Also,
$c_m(w)\bmod m=w_1$ when $w$ is nonempty.

The binary case gives the most transparent formula.  Write $c=c_2$, and for
a binary tag system put
\[
       \alpha:=c(P(0)),\qquad \beta:=c(P(1)).
\]

Allowing floor and remainder as operations, the whole update is neatly carried out by
\begin{equation}\label{eq:plain-tag-step}
 H_T(s)=\left\lfloor\frac{s}{2^v}\right\rfloor
 +2^{\lfloor\log_2s\rfloor-v}
 \left((s\bmod2)(\beta-\alpha)+\alpha-1\right).
\end{equation}
This formula can be understood piecewise:
\begin{enumerate}[leftmargin=2.2em]
\item $\lfloor s/2^v\rfloor$ shifts off the first (least significant) $v$ bits;
\item $\lfloor\log_2s\rfloor$ finds the marker bit, and hence where the new
production must begin;
\item $(s\bmod2 )(\beta - \alpha) + \alpha$ reads the first symbol, selecting $\alpha$ when it is $0$ and
$\beta$ when it is $1$.  The final $-1$ cancels the surviving word's old marker;
the appended production code supplies the new marker.
\end{enumerate}

Example: let $v=2$, $P(0)=001$, and $P(1)=0$. The two production
codes are $\alpha=12$ and $\beta=2$ (obtained by reversing $P(0)$ and $P(1)$
and prepending a $1$ bit).
Likewise, the word $00101$ has code $s=52$.
Equation~\eqref{eq:plain-tag-step} reads
\[
 \underbrace{\left\lfloor\frac{52}{4}\right\rfloor}_{13}
 +\underbrace{2^{5-2}}_{8}
  \underbrace{(12-1)}_{11}=101.
\]
In word form, this is $00101\mapsto101001$.  Applying the same update once
more gives $101\mapsto41$, corresponding to
$101001\mapsto10010$.

The theorem below is equation~\eqref{eq:plain-tag-step} with each forbidden
integer operation replaced using $J$ from Lemma~\ref{lem:floor}.
For readability, name the three replacements
\begin{align}
 Q_v(s)&:=J\!\left(\frac{s}{2^v}\right),
 &N(s)&:=J\!\left(\log_2s\right),\label{eq:binary-parts}\\
 D(s)&:=s-2J\!\left(\frac{s}{2}\right).&&
\end{align}
On encoded states these are, respectively, the shifted quotient, the word
length, and the first bit.

\begin{theorem}[The base-$2$ formula]\label{thm:binary-step}
For a binary tag system $T=(v,\{0,1\},P)$, define
\begin{equation}\label{eq:tag-step}
 F_T(s):=Q_v(s)+2^{N(s)-v}
 \left(D(s)(\beta-\alpha)+\alpha-1\right).
\end{equation}
For every binary word $w$ with $\len(w)\ge v$,
\[
       F_T(c(w))=c(\operatorname{step}_T(w)).
\]
\end{theorem}

\begin{proof}
Let $s=c(w)$ and $n=\len(w)$.  Lemma~\ref{lem:floor} gives
\[
 Q_v(s)=\left\lfloor\frac{s}{2^v}\right\rfloor,
 \qquad N(s)=n,
\]
and
\[
 D(s)=s-2\left\lfloor\frac{s}{2}\right\rfloor=w_1.
\]
Thus $D(s)(\beta-\alpha)+\alpha=c(P(w_1))$.  The quotient
$\lfloor s/2^v\rfloor$ deletes the first $v$ bits and leaves the marker at
position $n-v$.  Adding
\[
       2^{n-v}\bigl(c(P(w_1))-1\bigr)
\]
replaces that marker by the production and its new marker.  This is exactly
the base-$2$ encoding of the updated word.
\end{proof}

The same bookkeeping works for larger alphabets.  The only extra issue is
selecting among more than two productions.  Since the rule table is finite,
it can be written as one polynomial: let $L_T$ be the rational Lagrange
interpolation polynomial determined by
\begin{equation}\label{eq:lagrange}
  L_T(a)=c_m(P(a)),\qquad 0\le a<m,
\end{equation}
and write
\begin{align}
 D_m(s)&:=s-mJ\!\left(\frac{s}{m}\right),
 \label{eq:first-digit}\\
 N_m(s)&:=J\!\left(\log_ms\right).
 \label{eq:length-term}
\end{align}

\begin{proposition}[General alphabets]\label{prop:general-step}
For $T=(v,\Sigma,P)$ define
\begin{equation}\label{eq:general-step}
 F_{T,m}(s):=
 J\!\left(\frac{s}{m^v}\right)
 +m^{N_m(s)-v}\bigl(L_T(D_m(s))-1\bigr).
\end{equation}
For every word $w$ with $\len(w)\ge v$,
\[
    F_{T,m}(c_m(w))=c_m(\operatorname{step}_T(w)).
\]
\end{proposition}

\begin{proof}
Let $s=c_m(w)$ and $n=\len(w)$.  Lemma~\ref{lem:floor} gives
\[
 J\!\left(\frac{s}{m^v}\right)
 =\left\lfloor\frac{s}{m^v}\right\rfloor,
 \quad N_m(s)=n,
 \quad D_m(s)=s-m\left\lfloor\frac{s}{m}\right\rfloor=w_1.
\]
Hence $L_T(D_m(s))=c_m(P(w_1))$.  As in the binary proof, the first term
deletes $v$ digits and leaves the marker at position $n-v$, while
\[
        m^{n-v}\bigl(c_m(P(w_1))-1\bigr)
\]
replaces that marker by the production and its new marker.
\end{proof}

The halting threshold is analogous to the binary case:
\begin{equation}\label{eq:threshold}
     \len(w)<v \quad\Longleftrightarrow\quad c_m(w)<m^v.
\end{equation}
The forward implication follows from the upper bound in
\eqref{eq:length-bounds}, and the reverse implication from its lower bound.
The step term is used only while the state is at or above this threshold.

\section{From one step to unbounded computation}\label{sec:iteration}

Our machinery having been defined, we finally get to iterate. Evaluate $F_T$ once
for one tag step, then feed its output back into the same $F_T$. To be precise:

\begin{definition}\label{def:orbit-sim}
A partial map $F$ \emph{orbit-simulates} a tag system $T$ under an encoding
$c$ if
\[
       F(c(w))=c(\operatorname{step}_T(w))
\]
for every nonhalting configuration $w$.  A term language is
\emph{computationally universal under iteration} if it contains a fixed term
that orbit-simulates a fixed universal tag system under a computable input
encoding.
\end{definition}

This describes the simulated orbit only up to halting.  A fixed finite set of
subthreshold states cannot also serve as an unbounded output tape.

\begin{corollary}[Universality under iteration]\label{cor:orbit-universal}
Principal-value log--exp terms are computationally universal under iteration.
The terms may be chosen without constants.
\end{corollary}

\begin{proof}
Choose a known universal tag-system simulator by applying the Cocke--Minsky
construction to a fixed universal two-symbol Turing machine \cite{cocke}.
One detail matters here: a simulator may signal a halt with a special letter
without immediately becoming a short word.  Robinson observes that the
productions used after that signal can instead be made length-reducing, so
that the word reaches length $1$ exactly when the simulated machine halts
\cite[p.~195]{robinson}.  This gives a fixed tag system $T_*$ whose
natural short-word halting set is undecidable.
Proposition~\ref{prop:general-step} supplies its orbit-simulating source
expression, and Lemma~\ref{lem:compile} compiles that expression into a
constant-free log--exp term on the entire pre-halting state domain.
\end{proof}

In plain terms, no algorithm can decide whether repeated application of one
fixed term will eventually fall below one fixed threshold.

\begin{corollary}[A fixed undecidable reachability problem]
There are a constant-free log--exp term $F$ and an integer $B\ge4$ such that
\begin{equation}\label{eq:reachability-set}
 \left\{s\in\mathbb N:s\ge B,\ \begin{array}{l}
  \text{for some $t\ge1$, $F^j(s)\ge B$ for $0\le j<t$,}\\[-2pt]
  \text{and $F^t(s)<B$}
 \end{array}\right\}
\end{equation}
is undecidable.
\end{corollary}

\begin{proof}
Fix the tag system $T_*=(v,\Sigma,P)$ just described, and write
$m=|\Sigma|$.  Take the compiled term from
Proposition~\ref{prop:general-step} for $F$ and set $B=m^v$.  On every
integer $s\ge B$, the floor gadgets return an integer quotient, length,
and residue; the interpolation value is a production encoding.  Thus $F(s)$
is a positive integer, so the stopped orbits in
\eqref{eq:reachability-set} are well defined.  Now let
\[
       C:=\{c_m(w):w\in\Sigma^*,\ \len(w)\ge v\}.
\]
The decidable set $C$ consists of integers at least $B$ whose leading
base-$m$ digit is $1$.  Equations \eqref{eq:general-step} and
\eqref{eq:threshold} identify
natural halting on $C$ with the first crossing below $B$.  Excluding the
finitely many initially halted words does not change undecidability.  The
conditions in \eqref{eq:reachability-set} stop the orbit at that first
crossing; the compiled term is never evaluated afterward.  Consequently a
decision procedure for the full set \eqref{eq:reachability-set} would, by
restriction to $C$, decide the halting problem of $T_*$. 
\end{proof}

\section{Removing the arithmetic scaffolding}\label{sec:compile}

Equation~\eqref{eq:tag-step} still appears to cheat, as it clearly contains
constants and arithmetic operators.  With a little bit of exponent-juggling, we see
they turn out to be scaffolding, not extra primitives.  Fixing any positive base $N>1$ lets powers store
numbers in exponents and logarithms retrieve them.  The following identities
give a mechanical translation, so the enormous fully expanded term never
needs to be printed.

\begin{lemma}[Compilation]\label{lem:compile}
Fix a real $N>1$ and define
\[
  \mathbf{1}_N:=\log_N N,
  \qquad
  \boldsymbol{-1}_N:=\log_N\!\left(\log_{N^N}N\right).
\]
Then $\mathbf{1}_N=1$, $\boldsymbol{-1}_N=-1$, and, for real $x,y$ on the
indicated nonzero-denominator domains,
\begin{align}
  xy&=\log_N\!\left((N^x)^y\right),
  \label{eq:compile-mul}\\
  x+y&=\log_N\!\left(
       \log_N\!\left((N^{N^x})^{N^y}\right)\right),
  \label{eq:compile-add}\\
  x^{-1}&=x^{\boldsymbol{-1}_N},
  \qquad
  x/y=\log_N\!\left((N^x)^{y^{\boldsymbol{-1}_N}}\right).
  \label{eq:compile-div}
\end{align}
Together with \eqref{eq:J-pure}, these identities compile every source
formula in this paper into a term over $\{x^y,\log_ba\}$.

For the tag-system terms, no constants are required. Substitute the input
variable $s$ for $N$ in the term.  The resulting compiled term is defined and equal to the
source formula for every integer $s\ge m^v$, and thus in particular for
every valid state.
\end{lemma}

\begin{proof}
\medskip\noindent\textit{Why the identities work.}
Because $N>1$, all powers of $N$ occurring in
\eqref{eq:compile-mul}--\eqref{eq:compile-div} are positive real numbers, so
principal logarithms introduce no branch correction.  The identities follow
by direct calculation.  For example,
\[
 \log_{N^N}N=\frac1N,
 \qquad
 \log_N(1/N)=-1,
\]
so even the current state itself can manufacture the constant $-1$; for the
first state of the running example, $N=s=52$.  Also,
\[
 (N^{N^x})^{N^y}=N^{N^xN^y}=N^{N^{x+y}},
\]
which proves \eqref{eq:compile-add}.  Subtraction is addition after
multiplication by $-1$.  Starting from $1$ and $-1$, binary doubling and addition make
every integer; inversion and compiled multiplication make every rational.
Replacing the arithmetic operations one at a time therefore converts every
displayed source expression into powers and logarithms.  Before doing so, each
fixed interpolation polynomial is written in Horner form, so its evaluation
uses only addition and multiplication and is defined even when the selected
digit is $0$.

\medskip\noindent\textit{Why no constants need be assumed.}
For the constant-free tag-system version, take any integer input
$s\ge m^v\ge4$.  Lemma~\ref{lem:floor} makes the three quotient, residue,
and length gadgets real integers, and the interpolation polynomial is real.
Taking the transport base to be the input subterm $N=s$
therefore meets $N>1$ throughout the orbit up to and including the step that
crosses the threshold.  The values $1,-1$ and all machine constants are then
constant-valued terms built relative to $s$; essentially, constants are functions
of $s$ we can construct which return the same constant throughout this domain.  All source subterms are real on every such
integer input except for the unit-circle power occurring inside
\eqref{eq:J-pure}.
That power is nonzero and is passed directly to the primitive logarithm with
base $-1$; its logarithmic quotient is real.  Every displayed division has a
nonzero denominator, so the replacement stays within the domains just proved.
\end{proof}

\appendix
\section{A separate composition consequence}\label{sec:kalmar}

This section is not used in the tag-system simulation.  It answers the other
question raised in the introduction: what ordinary functions can be obtained
by one finite substitution term, without iteration?

Let $\E$ denote the standard class of Kalm\'ar elementary functions on $\N$.
The only fact about the class needed here is the basis theorem of Prunescu,
Sauras-Altuzarra, and Shunia.  In their substitution convention, constants
and projections are available, and their Theorem~1 states
\begin{equation}\label{eq:psa-basis}
    \E=\langle x+y,\ x\bmod y,\ 2^x\rangle,
\end{equation}
where $x\bmod0=x$ \cite{psa}.  Thus it is enough to represent these three
operations on natural-number inputs.

\begin{proposition}\label{prop:basis-terms}
Each operation in \eqref{eq:psa-basis} is the restriction to natural-number
inputs of a log--exp term with principal values and constants.
\end{proposition}

\begin{proof}
The third operation is already a power.  Addition is represented by
\begin{equation}\label{eq:addition}
 A(x,y):=\log_2\!\left(
     \log_2\!\left((2^{2^x})^{2^y}\right)\right)=x+y.
\end{equation}
All bases and logarithm arguments here are positive.

For total remainder, define
\begin{equation}\label{eq:zero-indicator}
  \delta(y):=J\!\left(\frac{3}{2(y+1)}\right),
  \qquad q(y):=y+\delta(y).
\end{equation}
On $\N$, $\delta(0)=1$ and $\delta(y)=0$ for $y\ge1$.  Hence $q(0)=1$,
$q(y)=y$ for $y\ge1$, and the denominator in
\begin{equation}\label{eq:remainder}
 R(x,y):=x-y\,
 J\!\left(\frac{x}{q(y)}\right)
\end{equation}
never vanishes.  Equation~\eqref{eq:remainder} gives $R(x,0)=x$; for $y\ge1$,
Lemma~\ref{lem:floor} gives
$R(x,y)=x-y\lfloor x/y\rfloor=x\bmod y$.  Equation~\eqref{eq:J-pure} and
Lemma~\ref{lem:compile} remove the displayed field operations.

Concretely, at $y=0$ the indicator changes the would-be denominator from $0$
to $1$, while the outside factor $y$ makes the correction vanish.  At $y=3$
the indicator is $J(3/8)=0$, so the same expression is just ordinary division
with remainder; for example, $R(11,3)=2$.
\end{proof}

\begin{corollary}\label{cor:kalmar}
Every Kalm\'ar elementary function is the restriction to $\N^r$ of a
log--exp term with principal values and constants.
\end{corollary}

\begin{proof}
Substitute the three representatives from Proposition~\ref{prop:basis-terms}
into a term supplied by \eqref{eq:psa-basis}.
\end{proof}

This is an inclusion statement, not a characterization of all
integer-valued restrictions of principal-value log--exp terms.
Constants are part of the cited substitution convention, so one may take
$N=2$ in Lemma~\ref{lem:compile}; the input-dependent choice $N=s$ is needed
only for the constant-free tag-step terms.

\section{A cyclic-tag formulation}\label{sec:cyclic}

In comments on the original Mathematics Stack Exchange post, r.e.s.
suggested a three-function computational basis---an implementation of BCT---using
the same binary encoding. I was able to express that basis within my framework.
A cyclic tag system is a list $(P_0,\ldots,P_{k-1})$ of binary words, with
$k\ge1$.
A step deletes the front bit, appends $P_j$ if that bit was $1$, and advances
$j\mapsto j+1\pmod k$; it halts when the dataword is empty \cite{cook}.

Pack a configuration $(w,j)$, with $0\le j<k$, as $S=kc_2(w)+j$.  Set
\begin{align*}
 j(S)&:=S-kJ\!\left(\frac{S}{k}\right),
 &s(S)&:=\frac{S-j(S)}{k},\\
 d(S)&:=s(S)-2J\!\left(\frac{s(S)}2\right),
\end{align*}
and let $L$ be the rational Lagrange interpolation polynomial satisfying
$L(j)=c_2(P_j)$ for $j=0,\ldots,k-1$.  For a nonempty word, put
\begin{equation}\label{eq:cyclic}
\begin{split}
 \Phi(S):={}&k\left[
   J\!\left(\frac{s(S)}2\right)
   +2^{J(\log_2s(S))-1}d(S)(L(j(S))-1)
   \right]\\
 &+j(S)+1-kJ\!\left(\frac{j(S)+1}{k}\right).
\end{split}
\end{equation}
Lemma~\ref{lem:floor} gives $j(S)=S\bmod k$ and makes the last line
$(j+1)\bmod k$.  The bracketed expression repeats the calculation in the
proof of Theorem~\ref{thm:binary-step}, now with one-symbol deletion and with
the production contribution multiplied by the deleted bit.  Thus
\eqref{eq:cyclic} is the exact cyclic-tag update.  The empty production has
encoding $1$ and contributes zero; a deleted $0$ also contributes zero.
Halting is the threshold $S<2k$.  Lemma~\ref{lem:compile}, with $N:=S$, makes
\eqref{eq:cyclic} a constant-free two-operation term on its pre-halting
domain.

\section{A direct Collatz step}\label{sec:collatz}

The principal-value convention also gives a short expression for the
\emph{shortcut Collatz map}
\begin{equation}\label{eq:collatz-step}
 T(s):=\begin{cases}
   s/2,&s\text{ even},\\
   (3s+1)/2,&s\text{ odd}.
 \end{cases}
\end{equation}
Here an odd step includes the division by $2$ that immediately follows
$3s+1$.  The input $s$ is the positive integer itself.

\begin{proposition}\label{prop:collatz}
For every integer $s>1$, the constant-free log--exp term
\begin{equation}\label{eq:collatz-pure}
 C(s):=
 \log_{\log_s((s^s)^s)}\!\left[
   \log_s\!\left(
     \left(s^{s^s}\right)^{
       \left[
         \log_s\!\left(\left((s^s)^{s^s}\right)^{s^s}\right)
       \right]^{
         \log_{\log_s(\log_{s^s}s)}\!\left(
           \left[\log_s(\log_{s^s}s)\right]^s
         \right)
       }
     }
   \right)
 \right]
\end{equation}
is defined and equals $T(s)$.  It contains $23$ binary operations:
$14$ powers and $9$ logarithms, with $s$ as its only leaf symbol.
\end{proposition}

\begin{proof}
To read the expression, temporarily abbreviate
\begin{align*}
 D&:=\log_s((s^s)^s)=s^2,\\
 M&:=\log_s(\log_{s^s}s)=-1,\\
 V&:=\log_s\!\left(\left((s^s)^{s^s}\right)^{s^s}\right)
     =s^{2s+1},\\
 p&:=\log_M(M^s)=
   \begin{cases}0,&s\text{ even},\\1,&s\text{ odd}.
   \end{cases}
\end{align*}
The identity for $M$ is the constant construction in
Lemma~\ref{lem:compile}.  The parity selector $p$ uses
$\Log(-1)=i\pi$: $M^s$ is $1$ or $-1$ according to the parity of $s$.
Substitution in \eqref{eq:collatz-pure} now gives
\[
 C(s)=\log_D\!\left(\log_s\!\left((s^{s^s})^{V^p}\right)\right)
     =\log_{s^2}\!\left(s^{s+(2s+1)p}\right)
     =\frac{s+(2s+1)p}{2}=T(s).
\]
All logarithm bases are among $s$, $s^s$, $D=s^2$, and $M=-1$;
none is $0$ or $1$ when $s>1$.  All logarithm arguments and power bases
are positive, except for the allowed values $M=-1$ and $M^s=-1$.
Thus every subterm is defined.  The positive-real powers used in the
simplification introduce no branch correction.
\end{proof}

Iteration is stopped when the value reaches $1$, since evaluating
\eqref{eq:collatz-pure} at $s=1$ would require a logarithm with base $1$.
The Collatz conjecture is equivalent to termination, for every positive
integer initial value, of the iteration of $C$ stopped upon reaching $1$.
This gives an exact expression for each step, without establishing that
every positive-integer orbit reaches $1$.

\paragraph{Acknowledgments and use of AI tools.}
I thank r.e.s. for the suggestions
that led me to explore the cyclic formulation, and Mihai Prunescu and Lorenzo
Sauras-Altuzarra for helpful correspondence and for pointing me toward the
Kalm\'ar basis question.  During preparation of this version I used
Anthropic's Claude, OpenAI's Codex, and ChatGPT to help check edge cases,
simplify the floor construction, and revise proofs and exposition.
All errors are my own.

\end{document}